\documentclass[10pt]{article}
\usepackage{latexsym,enumerate}
\usepackage{graphics,color,amsmath,amsthm,amsopn,amstext,amscd,amsfonts,amssymb}
\usepackage{vuk,a4wide,url,xspace}
\usepackage[american]{babel}
\usepackage{setspace}
\author{Eric Bonnetier\footnotemark[1] \and Didier Bresch\footnotemark[2] \and Vuk Mili{{\v s}}i{{\'c}}\footnotemark[1]}
\title{{\em A priori }convergence estimates for a rough Poisson-Dirichlet problem
with natural vertical boundary conditions}

\begin{document}
\renewcommand{\thefootnote}{\fnsymbol{footnote}}
\footnotetext[2]{LAMA, UMR 5127 CNRS, Universit{\'e} de Savoie, 73217 Le Bourget du Lac cedex, FRANCE}
\footnotetext[1]{LJK-IMAG, UMR 5523 CNRS, 51 rue des Math{\'e}matiques, B.P.53, 38041 Grenoble cedex 9, FRANCE}
\renewcommand{\thefootnote}{\arabic{footnote}}

\maketitle
\begin{center}
{\em "Dedicated to Professor Giovanni Paolo Galdi 60' Birthday"}
\end{center}

\begin{abstract}
Stents are medical devices designed to modify blood flow in aneurysm sacs, in order to prevent their rupture. Some of them can be considered as  a locally periodic rough boundary. In order to approximate blood flow in arteries and vessels of the cardio-vascular system  containing stents, we use multi-scale techniques to construct boundary layers and wall laws. Simplifying the flow we turn to consider a 2-dimensional Poisson problem that conserves essential features related to the rough boundary. Then, we investigate  convergence of boundary layer approximations and the corresponding wall laws in the case of Neumann type boundary conditions at the inlet and outlet parts of the domain. The difficulty comes from the fact that correctors, for the boundary layers  near the rough surface, may introduce error terms on the other portions of the boundary. In order to correct these spurious oscillations, we introduce a vertical boundary layer. Trough a careful study of its behavior, we prove rigorously decay estimates. We then construct  complete boundary layers that respect the macroscopic boundary conditions. We also derive error estimates in terms of the roughness size $\epsilon$ either for the full boundary layer approximation and for the corresponding  averaged wall law.
\end{abstract}

\bigskip


\bigskip


\bigskip

\section{Introduction}

A common therapeutic treatment to prevent rupture of aneurysms,
in large arteria or in blood vessels in the brain,
consists in placing a device inside the aneurysm sac.
The device is designed to modify the blood flow in this region,
so that the blood contained in the sac coagulates and the sac
can be absorbed into the surrounding tissue.
The traditional technique consists in obstructing the sac with a long coil. 
In a more recent procedure, a device called {\em stent}, that can
be seen as a second artery wall, is placed so as to close 
the inlet of the sac. 
We are particularly interested in stents produced by a company called Cardiatis,
which are designed as multi-layer wired structures.
Clinical tests show surprising bio-compatibility features of these particular devices
and one of our objectives is to understand how the design of these stents affect
their effectiveness.
As stent thicknesses are small compared to the characteristic dimensions 
of the flow inside an artery, studying their properties is a challenging
multi-scale problem.

In this work we focus on the fluid part and on the effects of the stent rugosity
on the fluid flow.
We simplify the geometry to that of a 2-dimensional box $\Oe$, that represents a
longitudinal cut through an artery: the rough base represents the shape of
the wires of the stent (see fig. \ref{Plaque}, left). 
We also simplify the flow model and consider a Poisson problem
for the axial component of the velocity.
Our objective is to analyze precisely multi-scale approximations of this
simplified model, in terms of the rugosity.

In \cite{BrMiQam} we considered periodic inflow and outflow boundary conditions
on the vertical sides $\gio$ of $\Oe$. 
Here, we study the case of more realistic Neumann  conditions on these boundaries,
which are consistent with the modelling of a flow of blood.

As a zeroth order approximation to $u^\epsilon$, we consider
the solution $\tuz$ of the same PDE, posed on a smooth domain 
$\Oz$ strictly contained in $\Oe$. 
We introduce boundary layer correctors $\beta$ and $\tau$ that correct the 
incompatibilities between the domain and $\tuz$. 
These correctors induce in turn perturbations on the vertical sides 
$\gio$ of $\Oe$. We therefore consider additional correctors
$\tin$ and $\tout$, that should account for these perturbations (see fig \ref{ShemaAWL_NP}). 
We also introduce a first order approximation, defined in $\Oz$, that
satisfies a mixed boundary condition (called Saffman-Joseph wall law)
on a fictitious interface $\Gz$ located inside $\Oe$.
\begin{figure}[h]
\begin{center}
\input{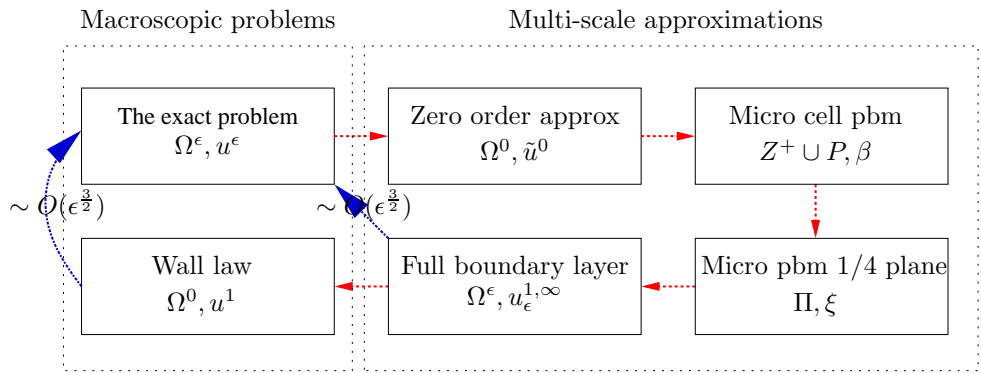}
\caption{The exact solution, the multi-scale framework and wall laws} \label{ShemaAWL_NP}
\end{center}
\end{figure}

For the case of Navier-Stokes equations and the Poiseuille flow
the problem was already considered in {\sc J{\"a}ger} {\em et al.} \cite{JaMiJDE.01,JaMiSIAM.00}
but the authors imposed  Dirichlet boundary conditions on $\gio$
for the vertical velocity and pressure. Their approach provided a localized 
vertical boundary layer in the $\epsilon$-close neighborhood of $\gio$.
A convergence proof for the boundary layer approximation and the wall law
was given wrt to $\epsilon$, the roughness size. These arguments are specific
to the case of Poiseuille flow and differ from the general setting
given in the homogenization framework \cite{SaPa.80}. 
In this work, we address the case of Neumann boundary conditions, 
where the above methods do not apply. 
The difficulty in this case, stems from the `pollution' on the
vertical sides due to the bottom boundary layer correctors.

From our point of view, the originality of this work emanate from the following aspects~:
\begin{itemize}
\item[-] the introduction of a general quarter-plane corrector $\xi$
that reduces the oscillations of the periodic boundary layer
approximations on a specific region of interest. Changing the
type of boundary conditions implies only to change the boundary
conditions of the quarter-plane corrector on a certain part
of the microscopic domain.
\item[-] the analysis of decay properties of this new corrector:
indeed we use techniques based on weighted Sobolev spaces
to derive some of the estimates and we complete this description
by integral representation and 
Fragmen-Lindel{\"o}f theory in order to derive sharper $L^\infty$ bounds.
\item[-] we show new estimates based on duality 
on the traces and provide a weighted
correspondence between macro and micro features of test
functions of certain Sobolev spaces.
\end{itemize}
The error between the wall-law and the exact solution is evaluated
on $\Oz$, the smooth domain above the roughness, in the $L^2(\Oz)$ norm.
This relies on very weak estimates \cite{Ne.Book.67} that
moreover improve {\em a priori} estimates by a $\sqrt{\epsilon}$ factor.
While this work focuses on the precise description of vertical boundary
layer correctors in the {\em a priori} part, a second article 
extends our methods to the very weak context \cite{vws} in order to obtain 
optimal rates of convergence also for this step.

The paper is organized as follows~: in section \ref{Section.Framework},
we present the framework (including notations, domains characteristics and the
toy PDE model under consideration), in section \ref{MswlP},
we give a brief summary of what is already available from the periodic
context \cite{BrMiQam} that should serve as a basis for what follows,
 in section \ref{MswlNP} we present a microscopic vertical boundary 
layer and its careful analysis in terms of decay at infinity, such decay 
properties will be used in section \ref{Section.Convergence}  in the convergence proofs 
 for the full boundary layer approximation as well as in the corresponding wall law
 analysis.

\section{The framework}\label{Section.Framework}

	In this work, $\Omega^\epsilon$ denotes  the rough domain in ${\mathbb R}^2$ depicted in fig. \ref{Plaque}, 
 $\Omega^0$ denotes the smooth one,   $\Gamma^\epsilon$ is the 
rough boundary and $\Gamma^0$ (resp. $\Gamma^1$) the lower (resp. upper) 
smooth one  (see fig \ref{Plaque}). 
The rough boundary $\Gamma^\epsilon$ is described as a periodic repetition
at the microscopic scale of a single boundary cell $P^0$. The latter 
can be parameterized as the graph of a Lipschitz function $f:[0,2\pi[\to ]-1:0[$, 
the boundary is then defined as
\begin{equation}\label{DefBoundaryPz}
P^0=\{ y \in [0,2\pi]\times]-1:0[ \text{  s.t.  }y_2= f(y_1) \}.
\end{equation}
Moreover we suppose that $f$ is bounded and negative definite, i.e. there exists a
positive constant $\delta$ such that $1-\delta < f(y_1)<\delta $ for all $y_1\in [0,2\pi]$.
The lower bound of $f$ is arbitrary and it is useful only in order to define 
some weight function see section \ref{MswlNP}.
We assume that the ratio between
$L$ (the width of $\Omega^0$) and $2\pi \epsilon$ (the width of the periodic cell) 
is always a positive integer.
\begin{figure}[h]
\begin{center}
\input{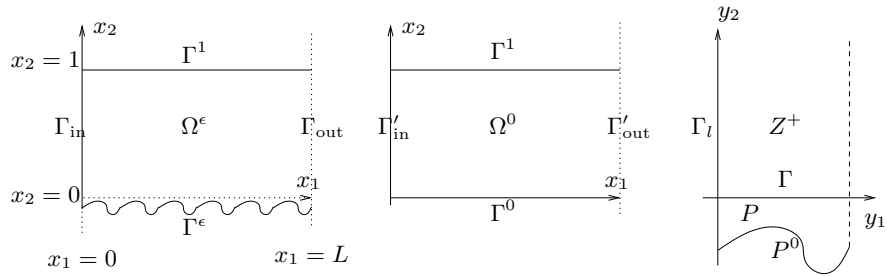}
\caption{\em Rough, smooth and cell domains} \label{Plaque}
\end{center}
\end{figure}
  We consider a simplified setting that avoids  theoretical difficulties and  non-linear complications of the full Navier-Stokes equations.
  Starting from the Stokes system, we consider a Poisson problem for  the axial component of the velocity.
The axial component of the pressure gradient is assumed to reduce to a 
constant right hand side $C$.
If we set   periodic inflow and outflow 
boundary conditions,
the simplified formulation reads~: find $u^\epsilon$ such that
\begin{equation}
\label{RugueuxComplet}
\left\{
\begin{aligned}
& - \Delta \uep = C,\text{ in  }  \Omega^\epsilon, \\
& \uep = 0,\text{ on }\Gamma^\epsilon \cup \Gamma^1, \\
& \uep \text{ is } x_1 \text{ periodic}.
\end{aligned}
\right.
\end{equation}
In section \ref{MswlP} we should give a brief summary of the framework 
already introduced in \cite{BrMiQam}. Nevertheless the main concern of this
work is to consider the  non periodic setting (see section \ref{MswlNP})
where we should consider an example of a more realistic inlet and outlet 
boundary conditions. Namely, we look for approximations of the 
problem: find $\uenp$ such that
\begin{equation}
\label{RugueuxCompletNP}
\left\{
\begin{aligned}
& - \Delta \uenp = C,\text{ in  } \Omega^\epsilon, \\
& \uenp = 0,\text{ on } \Gamma^\epsilon \cup \Gamma^1, \\
& \dd{\uenp}{\bf n} = 0,\text{ on } \gio.
\end{aligned}
\right.
\end{equation}
In what follows, functions that do depend on $y=x/\epsilon$ should be indexed by an $\epsilon$ (e.g. $\cu=\cu(x,x/\epsilon)$).

\section{Summary of the results  obtained in the periodic case \cite{BrMiQam}}\label{MswlP}

\subsection{The cell problems}
\subsubsection{The first order cell problem}
The rough boundary is periodic at the microscopic scale and 
this leads to solve the  microscopic cell problem:
 find $\beta$ s.t.
\begin{equation}
\label{A.cell}
\left\{
\begin{aligned}
& -\Delta \beta = 0,\text{ in } \zup,\\
& \beta = - y_2,\text{ on } P^0,\\
& \beta \text{ is } \yup.
\end{aligned}
\right.
\end{equation}
We define the microscopic average along the fictitious interface
$\Gamma$~: 
$
\obeta = \frac{1}{2\pi} \int_0^{2\pi} \beta(y_1,0) dy_1.
$
As $\zup$ is unbounded in the $y_2$ direction, we define also
$$
D^{1,2} = \{ v \in L^1_{\rm loc}(\zup)/\, D v \in L^2(\zup)^2, 
\, v \text{ is } y_1-\text{periodic }\},
$$ 
then one has the   result~:
\begin{thm}\label{exist_uniKsol_gamma} 
\label{prop.A.cell}
Suppose that $P^0$  is sufficiently smooth ($f$ is Lipschitz) and does not
intersect $\Gamma$. Let $\beta$ be a solution of \eqref{A.cell}, 
then it belongs to $D^{1,2}$. Moreover,
 there exists a unique periodic solution $\eta \in H^\ud(\Gamma)$, of 
the  problem
$$
<S\eta,\mu> = <1,\mu>, \quad \forall \mu \in H^\ud (\Gamma),
$$
where $<,>$ is the $(H^{-\ud}(\Gamma),H^\ud(\Gamma))$ duality bracket, and $S$  the inverse
of the Steklov-Poincar{{\'e}} operator.
One has the  correspondence between $\beta$ and the interface solution $\eta$~:
$$
\beta =H_{Z^+}\eta + H_P\eta,
$$
where $H_{Z^+}\eta$ (resp. $H_P\eta$) is the $y_1$-periodic harmonic extension of $\eta$ on $Z^+$ 
(resp. $P$). 
The solution in $Z^+$ can be written explicitly as a power series
of Fourier coefficients of $\eta$ and reads~:
$$
H_{Z^+} \eta = \beta(y)= \sum_{k=-\infty}^\infty  \eta_k e^{i k y_1 - |k| y_2},\quad \forall y \in Z^+,\quad \eta_k = \int_0^{2\pi} \eta(y_1)e^{-iky_1} dy_1,
$$
In the macroscopic domain $\Omega^0$ this representation formula gives
\begin{equation}\label{BetaL2}
\nrm{\beta\left(\frac{\cdot}{\epsilon}\right) - \obeta }{L^2(\Omega^0)} \leq K \sqrt{\epsilon} \nrm{\eta}{H^\ud(\Gamma)}.
\end{equation}
\end{thm}
\subsubsection{The second order cell problem}

The second order error on $\Geps$ should be corrected thanks to a
new cell problem~: find $\gamma \in D^{1,2}$  solving  
\begin{equation}\label{seKord_cell_pbm}
\left\{
\begin{aligned}
&-\Delta \tau = 0,\quad \forall y \in \zup, \\
& \tau = - y_2^2,\quad \forall y_2 \in P^0,\\
& \tau \text{ periodic  in }y_1. 
\end{aligned}
\right.
\end{equation}
Again, the horizontal average is denoted $\otau$. 
In the same way as for the first order cell problem,
one can obtain a  similar result:
\begin{prop}\label{exist.so_bl}
 Let $P^0$ be smooth enough and do not intersect $\Gamma$.
Then there exists a unique solution $\tau$ of \eqref{seKord_cell_pbm}
in $D^{1,2}(\zup)$. 
\end{prop}
\subsection{Standard averaged wall laws}
\subsubsection{A first order approximation}
Using the averaged value $\obeta$ defined above, one can construct a first order
approximation $u^1$ defined on the smooth interior domain $\Oz$ that solves~:
\begin{equation}
\label{macro_ordre_un}
\left\{
\begin{aligned}
& -\Delta u^1 = C,\quad \forall x \in \Omega^0,\\
& u^1 = \epsilon \obeta  \dd{u^1}{x_2},\quad \forall x\in \Gamma^0, \quad u^1 = 0,\quad \forall x\in \Gamma^1, \\
& u^1 \text{ is }\xup,
\end{aligned}
\right.
\end{equation}
whose explicit solution  reads~:
\begin{equation}\label{Poiseuille.order.one}
u^1(x)=-\frac{C}{2}\left( x_2^2 - \frac{x_2}{1+\epsilon \obeta} - \frac{\epsilon\obeta}{1+\epsilon \obeta} \right).
\end{equation}
Under the hypotheses of theorem \ref{prop.A.cell}, one  derives error estimates for the first order wall law
$$
\nrm{\uep - u^1}{L^2(\Oz)} \leq K \epsilon^\td.
$$
\subsubsection{A second order approximation} 
In the same way one should derive second order averaged wall law $u^2$ satisfying the  boundary value problem~:
\begin{equation}\label{macro_ordre_deux}
\left\{
\begin{aligned}
& -\Delta u^2 = C,\quad \forall x \in \Omega^0,\\
& u^2 = \epsilon \obeta \dd{u^2}{x_2} + \frac{\epsilon^2}{2} \otau \dd{^2u^2}{x_2^2},\quad \forall x \in \Gamma^0,\\
& u^2= 0,\quad \forall x \in \Gamma^1, u^2 \text{ is }\xup,
\end{aligned}
\right.
\end{equation}
whose solution exists, is unique \cite{BrMiQam} and writes~:
\begin{equation}\label{Poiseuille.order.two}
u^2(x)=-\frac{C}{2}\left( x_2^2 - \frac{x_2(1+\epsilon^2 \otau)}{1+\epsilon \obeta} - \frac{\epsilon(\obeta-\epsilon \otau)}{1+\epsilon \obeta} \right).
\end{equation}
Now, error estimates do not provide second order accuracy, namely we only obtain
$$
\nrm{\uep - u^2}{L^2(\Oz)} \leq K \epsilon^\td,
$$
which essentially comes from the influence of microscopic oscillations
that this averaged second order approximation neglects. Thanks 
to estimates \eqref{BetaL2}, one sees easily that these oscillations account
as $\epsilon^\td$ if not included in the wall law approximation.

\subsection{Compact form of the full boundary layer ansatz}
Usually in the presentation of wall laws, one first introduces 
the full boundary layer approximation. This approximation
is an asymptotic expansion defined on the whole 
rough domain $\Oe$. In a further step one averages this approximation
in the axial direction over a fast horizontal period and 
derives in a second step the corresponding standard wall law.

Thanks to various considerations already exposed in \cite{BrMiQam},
the authors showed that actually a reverse relationship could be
defined that expresses the full boundary layer approximations
as functions of the wall laws. Obviously this works because
the wall laws (defined only on $\Oz$) are explicit and 
thus easy to extend to the
whole domain $\Oe$. Indeed we re-define
\begin{equation}\label{ExtensionWallLawOrderOne}
u^1(x)=  \frac{C}{2} \left( ( 1 - x_2 ) x_2 \chiu{\Oz} +  x_2 \chiu{\oo} \right) 
- \frac{\epsilon\beta }{1+\epsilon \obeta} (1-x_2  ), \quad  \forall x \in \Oe
\end{equation}
while we simply extend $u^2$ using the formula \eqref{Poiseuille.order.two}
over the whole domain.
This leads to write~:
\begin{equation}
\label{fblcp}
\begin{aligned}
\uiuep &= u^1 + \epsilon \dd{u^1 }{x_2}(x_1,0)\left( \beta \lrxe - \obeta  \right),  \\
\uidep &= u^2 + \epsilon \dd{u^2 }{x_2}(x_1,0)\left( \beta \lrxe - \obeta  \right)  + \frac{\epsilon^2}{2} \dd{^2 u^2}{x_2^2} (x_1,0) \left(\tau  \lrxe - \otau  \right).
\end{aligned}
\end{equation}
For these first order and second order full boundary layer approximations
one can set the  error estimates \cite{BrMiQam}:
$$
\nrm{\uep - \uiue}{L^2(\Oz)} \leq K \epsilon^\td,\quad \nrm{\uep - \uide}{L^2(\Oz)} \leq K e^{-\frac{1}{\epsilon}}.
$$
Note that the second order full boundary layer approximation is very close
to the exact solution in the periodic case, an important step that this
work was aiming to reach is to show how far this can be extended to 
a more realistic boundary conditions considered in \eqref{RugueuxCompletNP}.
Actually, convergence rates provided hereafter and in \cite{vws}
show that only first order accuracy can be achieved trough the 
addition of a vertical boundary layer (see below). For this reason
we study in the rest of this paper only the first order full boundary
layer and its corresponding wall law.

\section{The non periodic case: a vertical corrector}\label{MswlNP}
The purpose of what follows is to extend above results
to the practical case of \eqref{RugueuxCompletNP}. We should show a general method to handle
such a problem. It is inspired in a part from the homogenization
framework already presented in \cite{SaPa.80,Vogelius} for
a periodic media in all directions. The approach below uses some arguments
exposed in \cite{BoVe} for another setting.

\subsection{Microscopic decay estimates}\label{Xi}

In what follows we  mainly  need to correct  oscillations
of the normal derivative of the first order boundary 
layer corrector $\beta$ on the inlet and outlet $\gio$.
For this sake, we define the  notations
$\Pi := \cup_{k=0}^{+\infty} [ \zup + 2 \pi k \eu]$, the vertical boundary will 
be denoted $E:=\{0\} \times ]f(0),+\infty[$ and  the bottom $B:=\{y\in P^0 \pm 2 k \pi \eu\}$ (cf. fig \ref{pi}).
In what follows we should denote $\Pi':=\rr^2$, $B':=\rr\times \{0\}$ and $E':=\{0\}\times \rr$.
\begin{figure}[h] 
\begin{center}
\input{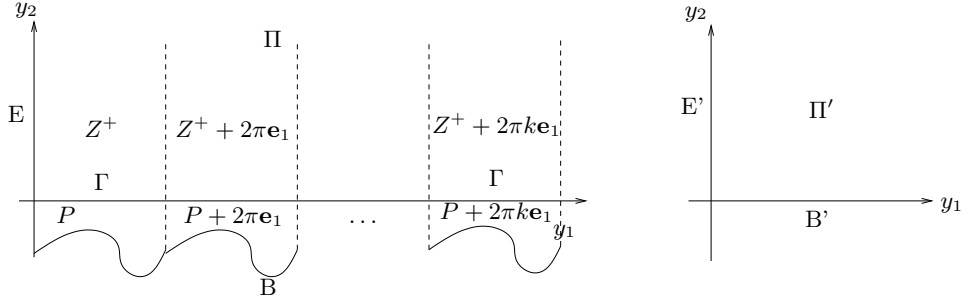}
\caption{\em Semi infinite microscopic domains: $\Pi$, the rough quarter-plane and $\Pi'$, the smooth one } \label{pi}
\end{center}
\end{figure} \\
On this domain,  we introduce the  problem: find $\xi$
such that
\begin{equation}
\label{tau}
\left\{
\begin{aligned}
& -\Delta \xi = 0, \quad \text{ in } \Pi, \\
& \dd{\xi}{\bf n}(0,y_2) = \dd{ \beta }{\bf n}(0,y_2) ,\quad \text{ on } E,\\
& \xi = 0, \quad  \text{ on } B.
\end{aligned}
\right.
\end{equation}
We  define the standard weighted Sobolev spaces~: for any given integers $(n,p)$ and a real $\alpha$ set
$$
\ws{n}{p}{\alpha}{\Omega}:=\left\{ v \in \cD'(\Omega) \, / \, |D^\lambda v| (1+\rho^2)^{\frac{ \alpha + |\lambda| - n}{2}{\nobreakspace}} \in L^p(\Omega),\, 0\leq |\lambda| \leq n \,\right\}
$$
where $\rho := \sqrt{ y_1^2 + (y_2+1)^2 }$. In what follows we should distinguish 
between properties depending on $\rho$ which is a distance to a point
exterior to the domain $\Pi$ and
$r=\sqrt{ y_1^2 + y_2^2 }$ the distance to the interior point $(0,0)$.
These weighted Sobolev spaces are Banach spaces for the norm
$$
\nrm{\xi}{W^{m,p}_\alpha(\Omega)}:= \left( \sum_{0\leq|\lambda|\leq m} \nrm{(1+\rho^2)^{\frac{\alpha-m+|\lambda|}{2}}D^\lambda u}{L^p(\Omega)}^p\right)^{\frac{1}{p}},
$$
the semi-norm being
$$
\snrm{\xi}{\ws{m}{p}{\alpha}{\Omega}}:= \left( \sum_{|\lambda|= m} \nrm{(1+\rho^2)^{\frac{\alpha-m+|\lambda|}{2}}D^\lambda u}{L^p(\Omega)}^p \right)^{\frac{1}{p}}.
$$
We refer to \cite{Ha.71,Ku.80.book,AmGiGiI.94} for detailed study of these
spaces. We introduce a specific subspace 
$$
\wso{p}{n}{\alpha}{\Pi}=\left\{v \in \ws{p}{n}{\alpha}{\Pi} \text{ s.t. } v\equiv0 \text{ on B}\right\}.
$$
We begin by some important properties satisfied by $\xi$ that will be used to 
prove convergence theorems \ref{5.1}, \ref{5.2} and \ref{5.3}.  
Such estimates will be obtained by a careful study of the 
weighted Sobolev properties of $\xi$ as well as its integral 
representation through a specific Green function.
\begin{thm}\label{DecayXi}
\uhthms, there exists $\xi$, a unique solution of problem \eqref{tau}.
Moreover $\xi \in \wso{1}{2}{\alpha}{\Pi}$ with $\alpha\in ]-\alpha_0,\alpha_0[$ where $\alpha_0:=(\sqrt{2}/ \pi)$ and 
$$
|\xi(y)| \leq \frac{K}{\rho(y)^{1-\frac{1}{2 M}}},\, \forall y \in \rr^2 \text{ s.t. } \rho>1, \quad \int_0^\infty \left| \dd{\xi}{y_1}(y_1,y_2) \right|^2 dy_1  \leq  \frac{ K}{ y_2^{1+2\alpha} },\, \forall y_2 \in \rr.
$$
where $M$ is a positive constant such that $M<1/(1-2\alpha)\sim10$.
\end{thm}
The proof follows  as a consequence of every result 
claimed until the end of subsection \ref{Xi}.
\begin{lem}\label{trace} In $\wso{1}{2}{\alpha}{\Pi}$ the semi-norm is a norm, moreover one has
$$
\nrm{\xi}{\ws{1}{2}{\alpha}{\Pi}}\leq \frac{1}{2 \alpha_0} \snrm{\xi}{\ws{1}{2}{\alpha}{\Pi}}, \quad \forall \alpha \in \R.
$$
On the vertical boundary $E$ one has the continuity
of the trace operator
$$
\nrm{\xi}{\ws{\ud}{2}{\alpha}{E}} \leq K \nrm{\xi}{\ws{1}{2}{\alpha}{\Pi}},\quad  \forall \alpha  \in \RR,
$$
the weighted trace norm being defined as
$$
\ws{\ud}{2}{0}{\partial \Pi}= \left\{ u \in \cD'(\partial \Pi) \text{ s.t. } \frac{u}{(1+\rho^2)^{\frac{1}{4}}} \in L^2(\partial \Pi),\, \int_{\partial \Pi_l^2} \frac{| u(y)-y(y')|^2}{|y-y'|^2} ds(y) ds(y') < + \infty \right\} ,
$$
and
$$
u \in \ws{\ud}{2}{\alpha}{\partial \Pi}  \iff (1+\rho^2)^{\frac{\alpha}{2}} u \in \ws{\ud}{2}{0}{\partial \Pi}.
$$
\end{lem}
The proof is omitted: the homogeneous Dirichlet condition on $B$ allows to establish Poincar{\'e} Wirtinger estimates (\cite{Galdi}, vol. I page 56) in a quarter-plane containing $\Pi$. Nevertheless, similar arguments are also used in the proof of lemma \ref{BordsBpPoid}.
\begin{lem}\label{BordL2Poid}
The normal derivative $g:=\dd{\beta}{y_1}(0,y_2)$ is a linear form on $\wso{1}{2}{\alpha}{\Pi}$ 
for every $\alpha\in\R$. 
\end{lem}
\begin{proof}
In $Z^+$, the upper part of the cell domain, the harmonic decomposition of $\beta$ 
allows to characterize its normal derivative explicitly on $E'$.
Indeed 
$$
g= \Re\left\{ \sum_{k=-\infty}^{+\infty} i k \eta_k e^{-|k| y_2} \right\} \chiu{E'}+ g_- =: g_+ + g_-
$$
where $g_-$ is a function whose support is located in $y_2\in [f(0),0]$. 
One has for the upper part
$$
\int_{E'}g_+^2 y_2^\alpha  dy_2 \leq K \nrm{\eta}{H^{\ud}(\Gamma)}^2, \quad \forall \alpha \in \R,
$$
thus $g_+$ is in the weighted $L^2$ space for any power of $(1+\rho^2)^\ud$, it is a linear form on $\ws{\ud}{2}{-\alpha}{E}$.
For $g_-$, we have no explicit formulation. We analyze the problem \eqref{A.cell}
but restricted to the bounded  sublayer $P$. We define $\beta_-$ to be harmonic
in $P$ satisfying $\beta_-=\eta$ on $\Gamma$, where $\eta$ is the trace on the fictitious interface
obtained in theorem \ref{prop.A.cell} and $\beta_-=-y_2$ on $P^0$. Note that thanks to standard regularity results $\eta \in C^2(\Gamma)$ because $\Gamma$ is strictly included in $\zup$, \cite{GiTru.Book}.
So $\beta$ solves a Dirichlet $y_1$-periodic problem in $P$ with regular data. As the
boundary is Lipschitz $\beta\in H^1(P)$ and so on the compact  interface $E_c:=\{0\}\times[f(0),0]$, $\partial_{\bf n}\beta$ is a linear form on $H^\ud$ functions. Then because $E_c$ is compact~:
$$
\int_{Ec} \dd{\beta}{\bf n} v dy_2
\leq \nrm{ \dd{\beta}{\bf n}}{H^{-\ud}(E_c)}\nrm{v}{H^{\ud}(E_c)}
\leq K \nrm{v}{\ws{\ud}{2}{\alpha}{E_c}},\quad \forall \alpha \in \R 
$$

\end{proof}

\begin{lem}\label{decay_estimates}
If $\alpha\in ]-\alpha_0:\alpha_0[$ 
there exists $\xi \in \wso{1}{2}{\alpha}{\Pi}$ a unique solution  of problem \eqref{tau}. 
\end{lem}
\begin{proof}
The weak formulation of problem \eqref{tau} reads
$$
(\nabla \xi,\nabla v)_{\Pi}= \left(g , v\right)_E,\quad \forall\, v  \in C^\infty(\Pi),
$$
leading to check hypothesis of the abstract inf-sup extension of 
the Lax-Milgram theorem \cite{Ne.Book.67,Ba.Siam.76}, for 
$$
a(u,v)=\int_{\Pi_+} \nabla u \cdot \nabla v\, dy, \quad l(v)=\int_{f(0)}^{+\infty} \dd{\beta}{\bf n} v \, dy_2.
$$
By lemma \ref{BordL2Poid}, $l$ is a linear form on $\ws{1}{2}{\alpha}{\Pi}$.
It remains to prove the inf-sup like condition on the bilinear form $a$. For this purpose
we set $v=u \rho^{2\alpha}$ and we look for a lower estimate of $a(u,v)$.
$$
a(u,u\rho^{2\alpha}) =\int_{\Pi_+} \nabla u \cdot  \nabla \left( u\rho^{2\alpha} \right) dy  =  \snrm{u}{\ws{1}{2}{\alpha}{\Pi}}^2 + {2\alpha} \int_{\Pi_+} \rho^{{2\alpha}-1}   u \nabla u\cdot \nabla \rho \, dy
$$
Using H{\"o}lder estimates one has
$$
\begin{aligned}
\int_{\Pi_+} \rho^{{2\alpha}-1}   u \nabla u \cdot \nabla \rho dy &\leq 
\left( \int_{\Pi_+} \rho^{{2\alpha}}  \left(\frac{u}{\rho}\right)^2 dy \right)^\ud  \left( \int_{\Pi_+} \rho^{{2\alpha}}  \left|\nabla u\right|^2 dy \right)^\ud \leq \frac{1}{{2\alpha}_0}\snrm{u}{\ws{1}{2}{\alpha}{\Pi}}^2
\end{aligned}
$$
In this way one gets
$$
a(u,u\rho^{2\alpha}) \geq (1 - \frac{\alpha}{\alpha_0} ) \snrm{u}{\ws{1}{2}{\alpha}{\Pi}}^2
$$
and if $\alpha< \alpha_0$ the inf-sup condition is fulfilled, the rest of the proof is standard and left to the reader \cite{Ba.Siam.76}.
\end{proof}
Thanks to the Poincar{\'e} inequality in $\Pi\setminus \Pi'$ 
with $\alpha=0$, we have
\begin{coro}
If a function $\xi$ belongs to $\wso{1}{2}{0}{\Pi}$  it satisfies $\xi \in L^2(B')$. 
\end{coro}
To characterize the weighted behavior of $\xi$ on $B'$ we set
$\omega_\alpha(y_1)=(y_1^2+1)^{\frac{2\alpha-1}{2}}y_1$ and we give
\begin{lem} \label{BordsBpPoid}If $\xi$ in $\wso{1}{2}{\alpha}{\Pi}$ then
$\xi \in L^2(B', \omega_\alpha):=\{ u \in \cD'(B') \text{ s.t. } \int_{B'} \xi^2 \omega_\alpha dy_1 < \infty\}$.
\end{lem}
\begin{proof}
$\Pi$ is contained in a set $\rr \times \{ - 1, +\infty\}$. We map the latter 
with cylindrical coordinates $(\rho,\theta)$. Every function of $\wso{1}{2}{0}{\Pi}$, extended
by zero on the complementary set of $\Pi$, belongs to the space of functions vanishing on the half-line $\theta=0$.
Using Wirtinger estimates, one
has for every such a function.
$$
\int_1^{+\infty} \xi^2 \left( \rho, \arcsin \left( \frac{1}{\rho} \right)\right) \rho^{2\alpha} d\rho \leq  \int_{1}^\infty \int_0^{\frac{\pi}{2}}\rho^{{2\alpha}-1} \arcsin\left( \frac{1}{\rho} \right) \left| \dd{\xi}{\theta } \right|^2 \rho d\theta d\rho \leq K \nrm{\xi}{\ws{1}{2}{\alpha}{\Pi}}^2
$$
because on $B'$, $\rho d\rho = y_1 d y_1$, one gets the desired result.
\end{proof}
In order to derive local and global $L^\infty$ estimates we introduce in this part
a representation formula of $\xi$ on $\Pi'$. As long as we use the representation formula below, $x$
will be the symmetric variable to the integration variable $y$. Until the end of proposition
\ref{deriv.horiz} both $x$ and $y$ are microscopic variables living in $\Pi$.

\begin{lem}\label{EstimXi}
 The solution of problem \eqref{tau} satisfies
$\xi(y) \leq K \rho^{-1+\frac{1}{2M}}$ for every $y \in \Pi'$ such that $\rho(y)\geq 1$.
The constant $M$ can be chosen such that $M<1/(1-{2\alpha})\sim 10$.
\end{lem} 
\begin{proof}
We set the representation formula
\begin{equation}\label{repres}
\xi(x) = \int_{E'} \gx g(y_2) dy_2 + \int_{B'} \dd{\gx}{\bf n} \xi(y_1,0) dy_1=: N(x) + D(x)  , \quad \forall x \in \Pi',
\end{equation}
where the Green function for the quarter-plane is 
$$
\gx(y) = \frac{1}{4\pi }\left( \ln |x-y| +\ln |x^*-y| -\ln |x_*-y| -\ln |\ov{x}-y| \right),
$$
with $x = (x_1,x_2),x^* = (-x_1,x_2),x_* = (x_1,-x_2),\ov{x} = (-x_1,-x_2)$.

\paragraph{  The Neumann part $N(x)$.}
We make the change of variables $x=(r \cos \vartheta ,r \sin \vartheta)$ which gives
$$\begin{aligned}
N&:= \lim_{m\to \infty} \sum_{k=0}^{m} \eta_k N_k
= \lim_{m\to \infty}\sum_{k=0}^{m} \frac{\eta_k}{2\pi } \int_0^\infty ke^{-ky_2} 
\left( 
\ln (x_1^2 + (y_2-x_2)^2) 
-  \ln (x_1^2 + (y_2+x_2)^2) 
\right) \, d y_2 \\
& =\lim_{m\to \infty} \frac{1}{2 \pi } \sum_{k=1}^m  \eta_k \int_0^\infty ke^{-ky_2}\left( 
 \ln \left(1-\frac{2 s y_2 }{r} + \left(\frac{y_2 }{r} \right)^2 \right)
-\ln \left(1+\frac{2 s y_2 }{r} + \left(\frac{y_2 }{r} \right)^2 \right) 
\right) dy_2 ,
\end{aligned}
$$
where $c=\cos \vartheta, s=\sin \vartheta $. Now we perform the second change of variables $t_k=e^{-k y_2}$ and get
$$
N_k \leq \frac{1}{\pi } \int_0^1 \ln \left ( 1 - \frac{2 s \ln t_k  }{r} + \left(\frac{\ln t_k }{r}\right)^2 \right) dt_k
\leq \frac{1}{\pi} \int_0^1 \ln\left(\left(1-\frac{\ln t}{r}\right)^2\right) dt 
$$
The last rhs is independent of $k$; one easily estimates it using the change of variables
$y=-\ln t/r$, indeed:
$$
\int_0^1 \ln\left(\left(1-\frac{\ln t}{r}\right)^2\right) dt = \int_0^\infty  \ln (1+y) e^{-r y} r dy = \int_0^\infty \frac{e^{-ry}}{1+y} dy \leq \frac{1}{r}
$$
Now because the fictitious interface $\Gamma$ is strictly included in $\zup$, $\beta \in H^2_{\loc}(\zup)$ 
and thus $\nrm{\eta}{H^1(\Gamma}:=\sum_{k=1}^\infty |\eta_k|^2 k^2 < + \infty$, one has for every finite $m$ 
$$
\begin{aligned}
\sum_{k=1}^m | \eta_kN_k | \leq \left( \sum_{k=1}^\infty  |\eta_k|^2 k^2 \right)^\ud \left( \sum_{k=1}^{\infty} \frac{1}{k^2} \right)^\ud  \leq C \nrm{\eta}{H^1(\Gamma)} \frac{1}{r}
\end{aligned}
$$
the estimate being uniform wrt $m$ one has
 that $N\leq C \nrm{\eta}{H^{1}(\Gamma)} / r$.

\paragraph{ The Dirichlet part $D(x)$.} We have, by the same change of variable as above ($x:=r(\cos(\vartheta),\sin \vartheta):=r(c,s)$~:
$$
\begin{aligned}
D(x) &= - \frac{x_2}{2\pi} \int_0^\infty  \left(  \frac{1}{(y_1-x_1)^2 + x_2^2} +    \frac{1}{(y_1+x_1)^2 + x_2^2} \right) \xi (y_1,0) d y_1 \\
& \leq \frac{s}{\pi r} \int_0^\infty \frac{1}{1 -2 c \frac{y_1}{r} + \left(\frac{y_1}{r}\right)^2 } | \xi | dy_1 
=  \frac{s}{\pi r} \int_0^\infty \frac{1}{(1 -c^2) +  \left(c-\frac{y_1}{r}\right)^2} | \xi | dy_1 ,
\end{aligned}
$$
where we suppose that $c<1$. We divide this integral in two parts, we set $m>1$
$$
D(x) \leq    \frac{ s}{\pi r} \left[ \int_0^m \frac{1}{(1 -c^2) +  \left(c-\frac{y_1}{r}\right)^2} | \xi | dy_1 +    \int_m^\infty \frac{1}{(1 -c^2) +  \left(c-\frac{y_1}{r}\right)^2} | \xi | dy_1 \right] =:(I_1 +I_2)(x) .
$$
For $I_1$ one uses the $L^p_{\rm loc} $ inclusions~:
$$ 
I_1 \leq  \frac{ s}{\pi r} \frac{K}{1-c^2} \nrm{\xi(\cdot,0) }{L^1(0,m)} \leq \frac{2}{\pi x_2} K \nrm{\xi }{L^2(B')},
$$
while for $I_2$ one uses the weighted norm established in  lemma \ref{BordsBpPoid}
$$
\begin{aligned}
I_2 &\leq  \frac{ s}{\pi r} \left( \int_m^\infty \left(\frac{1}{(1 -c^2) +  \left(c-\frac{y_1}{r}\right)^2}\right)^2 \frac{(y_1^2+1)^{\frac{1-{2\alpha}}{2}}}{ y_1 } dy_1 \right)^\ud \nrm{\xi}{L^2(B',\omega_\alpha)} \\
& \leq  \frac{2 s}{\pi r}  \left( \int_m^\infty \left(\frac{1}{(1 -c^2) +  \left(c-\frac{y_1}{r}\right)^2}\right)^2 y_1^{-{2\alpha}} dy_1 \right)^\ud \nrm{\xi}{L^2(B',\omega_\alpha)} \\
& \leq \frac{2 s}{\pi r}\left( \left(\int_m^\infty \left(\frac{1}{(1 -c^2) +  \left(c-\frac{y_1}{r}\right)^{2}}\right)^{2M}dy_1 \right)^{\frac{1}{M}} \left( \int_m^\infty  y_1^{-{2\alpha} M'} dy_1\right)^{\frac{1}{M'}} \right)^\ud \nrm{\xi}{L^2(B',\omega_\alpha)},
\end{aligned}
$$
where $M$ and $M'$ are H{\"o}lder conjugates.
We choose $M'{2\alpha} >1$ such that the weight contribution provided by $\xi$ is integrable, 
this implies that $M<1/(1-{2\alpha})\sim 10$. One then recovers
easily 
$$
I_2 \leq \frac{2 s K}{\pi r} \left( \frac{\pi r}{(1-c^2)^{2M-\ud}} \right)^{\frac{1}{2M}} = \frac{2K}{(\pi x_2)^{1-\frac{1}{2M}}}.
$$
We could shift the fictitious interface $\Gamma$ to $\Gamma - \delta e_2$ and repeat again the same arguments 
because the rough boundary does not intersect it. Note that in this case
we could establish again the  explicit Fourier representation formula for $\beta$ 
and its derivative as in theorem \ref{prop.A.cell}. Thus we can obtain 
that $\xi \leq c (x_2+\delta)^{1-1/(2M)}$ which shows that $\xi$ is bounded in $\Pi'$. 
So that on E', one has $|\xi| \rho^{1-\frac{1}{2M}} =|\xi| (1+x_2)^{1-\frac{1}{2M}} \leq c'$. Here one applies the Fragm{\`e}n-Lindel{\"o}f technique (see \cite{BoVe},  lemma 4.3, p.12). We restrict the
domain to a sector, defining $\Pi_S=\Pi\cap S$, where $S=\{ (\rho,\theta)\in [1,\infty]\times [0,\pi/2]\}$.
On $\Pi_S$, we define $\varpi:=-1+1/(2M)$ and $v:=\rho^{\varpi} \sin (\varpi \theta)$ the latter  is harmonic definite positive, we set $w=\xi/v$ which solves 
$$
\Delta  w + \frac{2}{v} \nabla v\cdot \nabla w=0, \text{ in }S_\Pi,  
$$
with  $w=0$ on $B_S =\partial S\cap B$, whereas $w$ is bounded uniformly on $E_S:=E\cap S$. 
Because by standard regularity arguments $\xi\in C^2(\Pi)$, $w$ is also bounded when $\rho=1$. 
Then by the Hopf maximum principle, we have 
$$
\sup_{\Pi_S} |w| \leq \sup_{\partial \Pi_S } w \leq K
$$
which extends to the whole domain $\Pi_S$, the radial decay of $\xi$.
\end{proof}

Deriving the representation formula \eqref{repres}, one gets for all $x$ strictly included
in $\Pi'$ that
$$
\partial_{x_1} \xi (x) = \int_{E'} \partial_{x_1} \Gamma_x  \, g dy_2 + \int_{B'} \dd{}{x_1} \dd{\Gamma_x}{y_2} \xi(y_1,0) dy_1 =: N_{x_1} (x) + D_{x_1}(x),\quad \forall x \in \Pi'
$$
\begin{lem}\label{EstimDerivXi}
For any $x\in \Pi'$, the Neumann part of the normal derivative $\partial_{x_1} \xi$ satisfies 
$N_{x_1}(x) \leq K r^{-2}$ for all $x$ in   $\Pi'$
\end{lem}
\begin{proof}
 On $E'$ the  derivative wrt $x_1$ of the Green kernel
reads
$$
\partial_{x_1} \Gamma_x = x_1 \left( \frac{1}{x^2_1 + (x_2-y_2)^2} - \frac{1}{x^2_1 + (x_2+y_2)^2} \right)
$$
thus using the cylindrical coordinates to express $x=(r\cos \vartheta,r\sin\vartheta )=:r(c,s)$ for $0\leq s<1$ one gets
\begin{equation}\label{esti_grosse}
\begin{aligned}
 N_{x_1} (x) &= 
\frac{c}{r} \int_0^\infty \left(\frac{1}{1-2s \frac{y_2}{r} +\left(\frac{y_2}{r} \right)^2} - \frac{1}{1+2s \frac{y_2}{r} +\left(\frac{y_2}{r} \right)^2} \right) g(y_2) dy_2 \\
& = \sum_k \frac{c}{r} \int_0^\infty \frac{4 s \frac{ y_2}{r}}{4s^2(1-s^2) +  \left(1-2s^2-\left(\frac{y_2}{r} \right)^2\right)^2} e^{-|k|y_2} dy_2\\
& \leq 4 \sum_k \frac{1}{x_1x_2} \int_0^\infty y_2 e^{-|k|y_2} dy_1 \leq \frac{1}{x_1x_2}\sum_k \frac{4}{k^2} |\eta_k|^2 \leq \frac{4}{x_1x_2} \nrm{\eta}{H^{-1}(\Gamma)}.
\end{aligned}
\end{equation}
This estimate is not optimal since it is singular near $x_1=0$ or $x_2=0$. But
it provides useful decay estimates inside $\Pi'$. 

It's easy to check that $N_{x_1}$ is harmonic, $N_{x_1}=g$ on $E'$, and that it
vanishes on $B'$. Because on $E'$ $g$ is bounded, by the maximum principle
$N_{x_1}$ is bounded.
We divide $\Pi'\setminus B(0,1)$ in three angular sectors~:
$$
S_i=\{ (r,\vartheta) \text{ s.t. } r>1, \, \vartheta \in \left[ \vartheta_{i-1},\vartheta_i \right] \},
\quad (\vartheta_i)_{i=0}^3 = \left\{ 0,  \frac{\pi}{6},\frac{\pi}{3},\frac{\pi}{2} \right\} 
$$
For $S_1$ and $S_3$ we define $v_i:=\pm \rho^{-2} \cos(2\vartheta),\, i=1,2$ which is positive definite
and harmonic, while for $S_2$,  we set $v_i:=\rho^{-2} \sin(2\vartheta)$, that shares the same properties.
For each sector we define $w_i=N_{x_1}/v_i$, it solves
$$
\Delta w_i + \frac{2}{v_i} \nabla v_i \cdot \nabla w_i = 0, \text{ in } S_i, \quad i=1,\dots,3
$$
The estimate \eqref{esti_grosse}  shows that on each interior boundary $\partial S_i$, $w$ is bounded
while on $E'\cup B'$ it is bounded by the boundary conditions that $N_{x_1}$ 
satisfies. By the Hopf maximum principle \cite{GiTru.Book}, one shows that
$$
\sup_{S_1} w_1 \leq \sup_{\vartheta=\frac{\pi}{6}} w_1 < \infty,\quad 
\sup_{S_2} w_2 \leq \sup_{\vartheta=\frac{\pi}{6},\vartheta=\frac{\pi}{3}} w_2 < \infty,\quad 
\sup_{S_3} w_3 \leq \sup_{\vartheta=\frac{\pi}{3},\vartheta=\frac{\pi}{2}} w_3 < \infty
$$
which implies that $N_{x_1}\leq K r^{-2}$ for every $y\in\Pi_S':=\Pi'\cap S$.
\end{proof}
In order to estimate $D_{x_1}$ the latter term of the derivative, we introduce
the  lemma inspired by proofs of weighted Sobolev imbeddings in  \cite{KudrI,KudrII}.
\begin{lem}\label{kudr.prop}
If $\xi \in \ws{1}{2}{\alpha}{\Pi}$ with $\alpha \in [0,1/2[$ then its trace on a horizontal
interface satisfies
\begin{equation}\label{kudr}
I(\xi)=\int_0^\infty \int_0^\infty \frac{\left|\xi(y_1+h,0)-\xi(y_1,0)\right|^2}{h^{2-{2\alpha}}} \, dy_1\, dh \leq \nrm{\xi}{\ws{1}{2}{\alpha}{\Pi}}
\end{equation}
\end{lem} 
The proof follows ideas of theorem 2.4' in \cite{KudrII} p. 235, we give it for sake of self-containtness.

\begin{proof}
We make a change of variables $x_1=r\cos \vartheta,x_2=r \sin \vartheta,\vartheta =0 $, leading to rewrite $I$ as
$$
I=\int_0^\infty \int_0^\infty \frac{\left|\xi(r+h,0)-\xi(r ,0)\right|^2}{h^{2-{2\alpha}}} \, dr  \, dh ,
$$
note that the second space variable for $\xi$ is now $\vartheta = 0$.
We insert intermediate terms  inside the domain, namely
$$
\begin{aligned}
I & \leq K \left\{ \int\int_0^\infty \frac{ \left| \xi (r+h,0) - \xi \left(r+h,\atan\frac{h}{r}\right) \right|^2}{h^{2-{2\alpha}}} +\frac{ \left| \xi \left(r+h,\atan\frac{h}{r}\right) - \xi \left(r,\atan\frac{h}{r}\right) \right|^2}{h^{2-{2\alpha}}}  dr  \, dh  \right. \\
& \left. + \int_0^\infty \int_0^\infty \frac{ \left| \xi \left(r,\atan\frac{h}{r}\right) - \xi \left(r,0\right) \right|^2}{h^{2-{2\alpha}}}  dr  \, dh \right\} =: I_1 + I_2 + I_1' \\
\end{aligned}
$$
Obviously the terms $I_1$ and $I_1'$ are treated the same way. We make a change of variable $(r,h=r \tan \vartheta )$ 
$$
I_1 = \int_0^{\frac{\pi}{2}} \int_0^\infty \frac{ \left| \xi (r(1+\tan \vartheta),0) - \xi (r(1+\tan\vartheta),\vartheta  ) \right|^2}{(r \tan \vartheta )^{2-{2\alpha}}}\frac{r}{1+\vartheta^2}  dr  \, d\vartheta .  \\
$$
In order to eliminate the dependence on $\vartheta$ in the first variable of $\xi$, we then make the change of variable
$(r=\tilde{r}/(1+\tan\vartheta),\vartheta)$ which gives
$$
I_1 = \int_0^{\frac{\pi}{2}}\int_0^\infty  \frac{ \left| \xi (\tilde{r},0) - \xi (\tilde{r},\vartheta  ) \right|^2}{\tilde{r}^{2-{2\alpha}}} \tilde{r} d\tilde{r}  \frac{(1+\tan\vartheta)^{{2\alpha}}}{\tan^{2-{2\alpha}} \vartheta (1 + \vartheta^2) } \,d\vartheta 
\leq  \int_0^\infty \int_0^{\frac{\pi}{2}} \frac{ \left| \xi (\tilde{r},0) - \xi (\tilde{r},\vartheta  ) \right|^2}{\tilde{r}^{2-{2\alpha}}} \tilde{r} d\tilde{r} \,\frac{d\vartheta}{\vartheta^{2-{2\alpha}}}.
$$
We are in the hypotheses of the Hardy inequality (see for instance \cite{KudrI}, p.203 estimate (7)), thus we have
$$
\begin{aligned}
I_1& \leq \frac{4}{(1-{2\alpha})^2} \int_0^\infty  \int_0^{\frac{\pi}{2}} \left| \dd{\xi}{\vartheta }(\tilde{r},\vartheta)  \right|^2 \vartheta^{{2\alpha}}d\vartheta  \tilde{r}^{{2\alpha} -1} d\tilde{r} \leq K  \int_0^\infty  \tilde{r}^{2\alpha} \int_0^{\frac{\pi}{2}} \frac{1}{\tilde{r}^2} \left| \dd{\xi}{\vartheta }(\tilde{r},\vartheta)  \right|^2 d\vartheta \tilde{r} d\tilde{r} \\
& \leq \nrm{\xi}{\ws{1}{2}{\alpha}{\Pi}}^2,
\end{aligned}   
$$
in the last estimate we used that in $\Pi'$, the distance to the fixed point 
$(0,-1)$ can be estimated as $\rho^{2\alpha}:=(y_1^2+(y_2+1)^2)^{\alpha}\geq (y_1^2+y_2^2)^{\alpha}=: \tilde{r}^{2\alpha}$, this explains 
why we need a positive $\alpha$ in the hypotheses.
In the same manner
$$
\begin{aligned}
I_2 & \leq   \int_0^{\frac{\pi}{2}} \int_0^\infty \frac{ \left| \xi (r(1+\tan \vartheta),\vartheta ) - \xi (r,\vartheta  ) \right|^2}{(r \tan \vartheta )^{2-{2\alpha}}}r  dr  \, d\vartheta, \\
& =  \int_0^{\frac{\pi}{2}} \int_0^\infty \left| \int_0^{r \tan \theta } \dd{\xi}{r} ( r + s ,\vartheta ) ds \right|^2 r^{{2\alpha}-1}  dr   \frac{d \vartheta }{\tan^{2-{2\alpha}}\vartheta },\\
& \leq \int_0^{\frac{\pi}{2}} \left\{ \int_0^{\tan \vartheta } \left[  \int_0^\infty  \left| \dd{\xi}{r} \right|^2 ( r(1 + \sigma ) ,\vartheta ) r^{ 1+{2\alpha}} dr \right]^\ud d \sigma \right\}^2  \frac{d \vartheta }{\tan^{2-{2\alpha}}\vartheta  },
\end{aligned}
$$
where we made the change of variables $s=r t$ and applied the generalized Minkowski inequality (\cite{KudrI}, p.203 estimate (6)). Now we set
$\tilde{r}=r(1+t)$ inside the most interior integral above
$$
I_2\leq  \int_0^{\frac{\pi}{2}} \left\{ \int_0^{\tan \vartheta }\frac{dt }{(1+t)^{1+\alpha}}  \left[  \int_0^\infty  \left| \dd{\xi}{r} \right|^2 (\tilde{r},\vartheta) \tilde{r}^{1+{2\alpha}} d\tilde{r}\right]^\ud \right\}^2  \frac{d \vartheta }{\tan^{2-{2\alpha}} \vartheta } .
$$
Now, we have  separated the integrals in $t$ and $r$, the part depending on $t$  is easy to integrate.
Thus we obtain
$$
I_2 \leq \int_0^{\frac{\pi}{2}} {\cal S}(\vartheta)   \int_0^\infty  \left| \dd{\xi}{r} \right|^2 (\tilde{r},\vartheta) \tilde{r}^{1+{2\alpha}} d\tilde{r}  d\vartheta ,\text{ where }  {\cal S}(\vartheta)= \frac{1}{\tan^{2-{2\alpha}}\vartheta }  \left[ \frac{1}{(1+\tan \vartheta )^{\alpha}} -1 \right]^2 
$$ 
Distinguishing whether $\tan \vartheta$ is greater or not than 1, it is possible to show that ${\cal S}$ is uniformly bounded wrt $\vartheta$. This gives
 the desired result.
\end{proof}

We follow similar arguments as in the proof of theorem 8.20 p. 144 in \cite{Kr.Book.99}
to claim:
\begin{prop}\label{deriv.horiz}
Set $\xi$ a function belonging to $\ws{1}{2}{\alpha}{\Pi}$, and 
$$
D_{x_1}(x) := \int_0^\infty  G(x,y_1) \xi(y_1,0) dy_1,\quad \forall x \in \Pi'  ,\text{ where } G(x,y_1)= \left.\dd{}{x_1} \dd{\Gamma_x }{y_2}\right|_{y\in B'},
$$
then it satisfies for every fixed positive $h$ 
$$
\int_0^\infty | D_{x_1}(x_1,h) |^2  dx_1 \leq \frac{K}{h^{1+{2\alpha}}},
$$
where the constant $K$ is independent on $h$.
\end{prop}
\begin{proof}
We recall that
$$
G:= - x_2 \left( \frac{x_1 - y_1}{ (( x_1 - y_1 )^2 + x_2^2 )^2} +  \frac{x_1 + y_1}{ (( x_1 + y_1 )^2 + x_2^2 )^2} \right).
$$
Because $\int_0^\infty G(x,y_1) dy_1 =0$ for every $x\in \Pi'$  we have
$$
D_{x_1} (x)= \int_0^\infty G(x,y_1) \left( \xi(y_1,0)-\xi(x_1,0) \right) dy_1,\quad \forall x \in \Pi',
$$
we underline that $G(x,\cdot)$ is evaluated at  $x\in \Pi'$ while $\xi(x_1,0)$ is taken on $B'$.
By  H{\"o}lder estimates in $y_1$ with $p=2,p'=2$, we have~:
\begin{equation}\label{HolderY}
| D_{x_1} |^2 \leq \int_0^\infty G^2 |y_1-x_1|^{2-{2\alpha} } dy_1 \int_0^\infty \frac{| \xi(y_1,0)-\xi(x_1,0)|^2}{|y_1-x_1|^{2-{2\alpha} }} dy_1
\end{equation}
integrating in $x_1$ and using H{\"o}lder estimates with $p=\infty,p'=1$ then
$$
I_3:= \int_0^\infty | D_{x_1} |^2 dx_1 \leq \sup_{x_1\in \rr} \int_{\rr} G^2 |y_1 -x_1 |^{2-{2\alpha}}\, dy_1 \int_0^\infty \int_0^\infty \frac{| \xi(y_1,0)-\xi(x_1,0)|^2}{|y_1-x_1|^{2-{2\alpha} }} dy_1 dx_1
$$
Thanks to proposition \ref{kudr.prop} we estimate the last integral in the rhs above
$$
I_3 \leq K \sup_{x_1\in \rr} I_4(x_1,x_2) \nrm{\xi}{\ws{1}{2}{\alpha}{\Pi}},\text{ where } I_4(x_1,x_2) :=  \int_{\rr} G^2 |y_1 -x_1 |^{2-{2\alpha}}\, dy_1
$$
Considering $I_4$ one has
$$
\begin{aligned}
I_4(x)& \leq  K \int_0^\infty \frac{x_2^2 ( y_1 -x_1)^{4-{2\alpha}}}{((x_1-y_1)^2+x_2^2)^4} dy_1 + \frac{x_2^2 ( y_1 +x_1)^2 (y_1-x_1)^{2-{2\alpha}}}{((x_1+y_1)^2+x_2^2)^4} dy_1 \\
& \leq K \int_0^\infty \frac{x_2^2 ( y_1 -x_1)^{4-{2\alpha}}}{((x_1-y_1)^2+x_2^2)^4} dy_1 + \frac{x_2^2 ( y_1 +x_1)^{4-{2\alpha}}}{((x_1+y_1)^2+x_2^2)^4} dy_1 =:I_5+I_6
\end{aligned}
$$
Both terms in the last rhs are treated the same, namely
$$
I_5 \leq x_2^{2+5-{2\alpha} - 8} \int_{-\infty}^\infty \frac{z^{4-{2\alpha}}}{(z^2+1)^4} dz \leq \frac{K}{x_2^{1+{2\alpha}}}
$$
which ends the proof.
\end{proof}
\begin{rmk} This is one of the key point estimates 
of the paper. One could think of using weighted properties of $\xi$ of lemma \ref{BordsBpPoid}
instead of the fractional Sobolev norm introduced from proposition \ref{kudr.prop},
in the H{\"o}lder estimates \eqref{HolderY}. This implies to transfer the $x_1$-integral
on $G$, then it seems impossible to conclude because
$$
\int_0^\infty \int_0^\infty G^2 \omega_\alpha(y_1)  dy_1 d x_1 = \infty,
$$
which is easy to show if one performs the  change of variables $z_1=y_1-x_1,y_1=y_1$.
\end{rmk}

In this part we study the convergence properties of the normal derivative of $\xi$
on vertical interfaces far from $E$. For this sake we call
$$
\Pi_l=\{ y \in \Pi,\text{ s.t. } y_1>l \},\quad E_l=\{ y_1=l, \, y_2 \in [f(0),+\infty[ \}, \quad B_l = \{ y\in B, y_1 > l\},
$$
here we redefine the weighted trace spaces of Sobolev type
$$
\begin{aligned}
\ws{\ud}{2}{0,\sigma}{\partial \Pi_l}= & \left\{ u \in \cD'(\partial \Pi_l) \text{ s.t. } \frac{u}{(1+\sigma^2)^{\frac{1}{4}}} \in L^2(\partial \Pi_l),\, \int_{\partial \Pi_l^2} \frac{| u(y)-y(y')|^2}{|y-y'|^2} ds(y) ds(y') < + \infty \right\} 
\end{aligned}
$$
where we define $\sigma:= |y-(0,f(0))|=\sqrt{y_1^2 + (y_2-f(0))^2}$. Note that the weight is  a distance
to the fixed point $(0,f(0))$ independent on $l$.
\begin{prop}\label{lift}
Suppose that $v\in \ws{\ud}{2}{0,\sigma}{\partial \Pi_l}$ with $v=0$ on $B_l$ then there
exists an extension  denoted $R(v)\in \cD'(\Pi_l)$ s.t.
$$
\snrm{\nabla R(v)}{L^2(\Pi_l)} \leq K \nrm{v}{\ws{\ud}{2}{0,\sigma}{\partial \Pi_l}},
$$
where $K$ depends only on $\nrm{f'}{\infty}$.
\end{prop}

\begin{proof}
We lift the domain in a first step in order to transform $\Pi_l$ in a quarter-plane $\hat{\Pi}_l$.
$$
y=\varphi(Y):=\begin{pmatrix} Y_1 \\ Y_2 + f(Y_1) \end{pmatrix}, \quad Y \in \hat{\Pi}_l := (\R_+)^2,
$$
if we set $\hat{v}(Y_2)=v(Y_2+f(0))=v(y_2)$ then the $L^2$ part of the weighted norm
above reads
$$
\begin{aligned}
\int_{E_l} \frac{v^2}{(1+\sigma^2)^\ud} dy_2 & = \int_{E_l}  \frac{v^2}{(1+l^2+(y_2-f(0))^2)^\ud} dy_2 
= \int_{\{l\}\times\R_+}  \frac{\hat{v}^2}{(1+l^2+Y_2^2)^\ud} dY_2 \\
& = \int_{\{l\}\times\R_+}  \frac{\hat{v}^2}{(1+\hat{\sigma}^2)^\ud} dY_2 ,
\end{aligned}
$$
where $\hat{\sigma}^2 = Y_1^2+Y_2^2$. If $v\in \ws{\ud}{2}{0,\sigma}{\partial \Pi_l}$ and $v=0$ on $B_l$ 
we know that (\cite{Grisvard}, p.43 theorem 1.5.2.3) 
$$
\int_0^\delta |\hat{v}(Y_2)|^2 \frac{d Y_2}{Y_2} < + \infty,
$$
which authorizes us to extend $v$ by zero on $\hat{E}_l:=\{Y_1=l\} \times \R$, this extension
still belongs to $\ws{\ud}{2}{0,\hat{\sigma}}{\hat{E}_l}$. 
Arguments above allow obviously to write for every $v$ vanishing 
on $B$ and $\hat{v}$ defined above
$$
\nrm{v}{\ws{\ud}{2}{0,\sigma}{\partial \Pi_l}} = \nrm{\hat{v}}{\ws{\ud}{2}{0,\hat{\sigma}}{\hat{E}_l}}.
$$
Here we use trace theorems II.1 and II.2 of {\sc Hanouzet} \cite{Ha.71}, they follow exactly the same in our case
except that the weight is not a distance to a point on the boundary (as in  \cite{Ha.71})
but it is a distance to a point exterior to the domain. So in order to define an extension (\cite{Ha.71} p. 249), we set
$$
\left\{ 
\begin{aligned}
V(Y)& = \int_{|s|<1} {\cal K}(s) \hat{v}(Y_1 s + Y_2 ) ds ,\quad s \in \R, \quad \forall Y \in \hat{\Pi}_l, \\
\Psi (Y)& = \Phi \left( \frac{Y_1-l}{(1+Y_2^2 + l^2 )^\ud} \right), 
\end{aligned}
\right.
$$
where $\Phi$ is a cut-off function such that
$$
{\rm Supp} \Phi \in \left[0:\uq\right[,\quad \Phi(0)=1,\quad \Phi \in C^\infty\left(\left[0,\uq\right[\right),
$$
and ${\cal K}$ is a regularizing kernel i.e. ${\cal K} \in C^\infty_0(]-1:1[)$
and $\int_\R {\cal K}(s) ds=1$.
Then the extension in the quarter-plane domain reads
$$
w(Y)= (\Psi V)(Y_1,Y_2) - (\Psi V)(Y_1,-Y_2) , \quad Y \in \hat{\Pi}_l,
$$
which allows to have $w(Y_1,0)=0$ for all $Y_1\in \R_+$. According to theorems II.1 and II.2
in \cite{Ha.71}, one then gets
$$
 \nrm{\frac{w}{(1+\hat{\sigma}^2)^\ud}}{L^2(\hat{\Pi}_l)} \leq K \nrm{ \hat{v}}{\ws{\ud}{2}{0,\sigma}{\hat{E}_l}},\quad  \nrm{\nabla w}{L^2(\hat{\Pi}_l)} \leq K \nrm{ \hat{v}}{\ws{\ud}{2}{0,\sigma}{\hat{E}_l}}.
$$
Turning back to our starting domain $\Pi_l$, we set 
$$
R(v)= w(\varphi^{-1}(y))= w(y_1,y_2-f(y_1)),\quad \forall y \in \Pi_l.
$$
We focus on the properties of the gradient
$$
\int_{\Pi_l} (A \nabla_y R(v), \nabla_y R(v)) dy = \int_{\hat{\Pi}_l} |\nabla_Y w|^2 d Y, \text{ where } A= \begin{pmatrix} 1 & f' \\
f' & 1+(f')^2 
\end{pmatrix},
$$
but the eigenvalues of $A$ are
$$
\lambda_\pm = \frac{ 2 + (f')^2 \pm \sqrt{ 2 + (f')^2} |f'| }{2},
$$
the lowest eigenvalue is positive and tends to zero as $|f'|$
increases. The boundary is Lipschitz so that $\nrm{f'}{\infty}$ is
bounded. Thus there exists a minimum value of $\lambda_-$. All this guarantees
the existence of a constant $\delta'(\nrm{f'}{\infty})>0$ such that
$$
\delta' \int_{\Pi_l} \left|\nabla_y R(v)\right|^2 dy \leq K \nrm{\hat{v}}{\ws{\ud}{2}{0,\hat{\sigma}}{\hat{E}_l}},
$$
which ends the proof.
\end{proof}

Thanks to the existence of a  lift $R(v)$, we are able to 
estimate a sort of weak weighted Sobolev norm for the normal
derivative on vertical interfaces located at $y_1=L/ \epsilon$. 
\begin{prop}\label{esitm.weight.sob.trace}
If $v \in \ws{\ud}{2}{0,\sigma}{\partial \Pi_{\frac{L}{\epsilon}}}$ and $v=0$ on $B_{\frac{L}{\epsilon}}$ then
one has
$$
\int_{E_{\frac{L}{\epsilon}}} \dd{\xi}{\bf n}\left(\frac{L}{\epsilon},y_2\right) v(y_2) dy_2 \leq K \epsilon^{\alpha} \nrm{ v }{\ws{\ud}{2}{0,\sigma}{\partial \Pi_{\frac{L}{\epsilon}}}}
$$
\end{prop}
\begin{proof}
The function $v$ given in the hypotheses belongs to the adequate spaces in order to
apply proposition \ref{lift}, thus there exists a lift $R(v) \in \ws{1}{2}{0}{\pile}$ s.t. $R(v)=v$
on $\partial \Pi_{\frac{L}{\epsilon}}$. Because $\xi$ is harmonic and belongs to $\ws{1}{2}{0}{\Pi}$,
for any $\varphi\in \D(\overline{\pile})$ and $\varphi_{|B_{\frac{L}{\epsilon}}}=0$~,
one writes the variational form:
$$
\int_{\pile} \nabla \xi \cdot \nabla \varphi dy = \int_{\partial \pile} \ddn{\xi } \varphi d\sigma(y) = \int_{E_{\frac{L}{\epsilon}}}\ddn{\xi } \varphi dy_2
$$
Then by density and continuity arguments one extends this formula
to every test functions in $\varphi\in \ws{1}{2}{0,\sigma}{\pile}$ such that
$\varphi=0$ on $B_{\frac{L}{\epsilon}}$. As the specific lift $R(v)$ belongs to this
space one has
$$
\begin{aligned}
\int_{E_{\frac{L}{\epsilon}}} \dd{\xi}{\bf n} v dy_2& 
= \int_{\Pi_{\frac{L}{\epsilon}}} \nabla \xi \nabla R(v) dy  
\leq \left( \sup_{\Pi_{\frac{L}{\epsilon}}} \frac{1}{\rho^{2\alpha}} \int_{\Pi_l} | \nabla \xi|^2 \rho^{2\alpha} dy \right)^\ud 
\nrm{\nabla R(v)}{L^2(\Pi_{\frac{L}{\epsilon}})} \\
& \leq K \epsilon^{\alpha} \nrm{\xi}{\ws{1}{2}{\alpha}{\Pi_{\frac{L}{\epsilon}}}}
    \nrm{\nabla R(v)}{L^2(\Pi_{\frac{L}{\epsilon}})}\leq K'  \epsilon^{\alpha} \nrm{v}{\ws{\ud}{2}{0,\sigma}{\partial \Pi_{\frac{L}{\epsilon}}}}
\end{aligned}
$$
which ends the proof.
\end{proof}

\subsection{ Test functions: from macro to micro and vice-versa} \label{test}

We suppose that $v\in H^1_D(\Oe):=\{ u \in H^1(\Oe), u=0 \text{ on } \Geps \cup \Gun\}$ then 
$\gamma(v)\in H^{\ud}(\partial \Oe)$ which implies that $v\in H^\ud(\gio)$ and that for any corner
$$
\int_0^\delta |v(x(t))|^2 \frac{d t}{t} < \infty ,
$$
where $x(t)\in \gio $ is a mapping of the neighborhood of the corners.
To the trace of $v$ on $\Gin$ or $\Gout$, we associate a trace of a function defined on $\partial \Pi_{\frac{L}{\epsilon}}$
which is zero on $B_{\frac{L}{\epsilon}}$ s.t.
$$
\ti{v}\left( \frac{L}{\epsilon} , y_2\right):= v(0,\epsilon y_2) = v(0,x_2) ,\quad \forall x_2 \in [\epsilon f(0),1] \quad \text{ and } \ti{v}\left( \frac{L}{\epsilon} , y_2\right):= 0, \quad y_2 > \ue,
$$
then one has the following connexion between the macroscopic trace norm and the microscopic weighted one.
\begin{prop}\label{test.prop}
Under the hypotheses above on functions $v$ and $\ti{v}$,
$$
\nrm{v}{H^\ud(\Gin\cup\Geps\cup\Gun)} \sim \nrm{\ti{v}}{\ws{\ud}{2}{0,\sigma}{\partial \Pi_{\frac{L}{\epsilon}}}},
$$
if $v=0$ on $\Geps\cup\Gun$ (resp. $\ti{v}=0$ on $B_{\frac{L}{\epsilon}}$).
\end{prop}

\begin{proof}
Thanks to the change of variables $x_2=\epsilon y_2$ we have that
$$
\begin{aligned}
\int_{\epsilon f(0)}^1 v^2(0,x_2) dx_2 &
= \epsilon \int_{f(0)}^{\ue}\ti{v}^2 \left(\frac{L}{\epsilon},y_2\right) dy_2 
\leq K \epsilon  \sup_{ E_{\frac{L}{\epsilon}}} (1+\sigma^2)^\ud \int_{E_{\frac{L}{\epsilon}}} \frac{\ti{v}^2}{ (1+\sigma^2)^\ud} dy_2,  \\
& \leq K \epsilon {\frac{L}{\epsilon}} \nrm{\ti{v}}{\ws{\ud}{2}{0,\sigma}{\partial\Pi_{\frac{L}{\epsilon}}}}^2.
\end{aligned}
$$
Conversely
$$
\begin{aligned}
\int_{E_{\frac{L}{\epsilon}}} \frac{\ti{v}^2}{ (1+\sigma^2)^\ud} dy_2 & = \int_{f(0)}^{\ue}  \frac{\ti{v}^2}{ (1+\sigma^2)^\ud} dy_2 \leq\sup_{ E_{\frac{L}{\epsilon}}} \frac{1}{(1+\sigma^2)^\ud  } \int_{f(0)}^{\ue}  \ti{v}^2 dy_2  \\
& \leq K \epsilon \nrm{\ti{v}}{L^2(f(0),\ue)}^2 = K \nrm{ v}{L^2(\Gin)}^2.
\end{aligned}
$$
For the semi-norm the same change of variable provides an equality 
due to the homogeneity in $\epsilon$ i.e.
$$
\snrm{v}{H^\ud(\Gin)}^2
= \int\int_{\Gout^2}  \frac{|v(x_2)- v(x_2)|^2}{|x_2-x_2'|^2} dx_2 dx_2' 
= \snrm{\ti{v}}{\ws{\ud}{2}{0,\sigma}{E_{\frac{L}{\epsilon}}}}^2
$$
\end{proof}

\begin{rmk}
We insist on the fact that one can associate traces of $v$ 
either from $\Gin$ or $\Gout$ to $\ti{v}$, the weight 
that one gains in the microscopic norm comes from the scaling from macro to 
micro and not from the vertical position of the macroscopic interface wrt the origin
of the domain $\Oe$.
\end{rmk}

\section{A new proof of convergence for standard averaged wall laws}\label{Section.Convergence}
\subsection{The full first order boundary layer approximation: error estimates}\label{FoblErrEst}
The periodic boundary layer approximations given in \eqref{fblcp} introduce
some microscopic oscillations on the inlet and outlet boundaries $\gio$.
We define a new full boundary layer approximation
\begin{equation}
\label{fblc}
\uiue = u^1 + \epsilon \dd{u^1 }{x_2}(x_1,0)\left( \beta  - \obeta  - \tin - \tout \right) \lrxe 
\end{equation}
where we define
$$
\tin\lrxe =\xi\lrxe,\quad 
\tout\lrxe =\tilde{\xi} \left(\frac{x_1-L}{\epsilon},\frac{ x_2 }{\epsilon} \right),
$$
and $\xi$ is  the solution of problem \eqref{tau}, and $\tilde{\xi}$ solves
the symmetric problem for $\Gout$:
$$
\left\{ 
\begin{aligned} 
& - \Delta \tilde{\xi} = 0,\quad \text{ in }\Pi_-:= \cup_{k=1}^\infty \{ \zup - 2\pi k e_1 \} \\
& \ddn{\tilde{\xi}}=  \ddn{\beta},\text{ on } E\\
& \tilde{\xi} = 0,\text{ on } B_-:= \cup_{k=1}^\infty \{ P^0- 2\pi k e_1 \}   
\end{aligned}  
\right. 
$$
Every result shown for $\xi$ in sections above holds equally
for $\tilde{\xi}$. 
One easily checks that 
$$
\left.\dd{\tin}{\bf n} \right|_{\Gout} = \ue \dd{\xi}{\bf n} \left(\frac{L}{\epsilon},\frac{x_2}{\epsilon}\right)
\text{ and }
\ddn{\tout}|_{\Gin} = \ue \ddn{\tilde{\xi}}\left(-\frac{L}{\epsilon},\frac{x_2}{\epsilon}\right)
$$
We estimate the error of this new boundary layer approximation. We denote
$\rui:= u^\epsilon - \uiue$, it solves 
\begin{equation}\label{sys.eq.error}
\left\{ 
\begin{aligned}
&- \Delta \rui = C \chiu{\oo}, \text{ on } \Oe,\\
& \dd{r_\epsilon}{\bf n}^{1,\infty} =   \dd{\xi}{\bf n}\left(\frac{L}{\epsilon},\frac{x_2}{\epsilon}\right) \text{ on } \Gout,  \dd{r_\epsilon}{\bf n}^{1,\infty} =    \dd{\tilde{\xi} }{\bf n}\left(\frac{L}{\epsilon},\frac{x_2}{\epsilon}\right) \text{ on } \Gin, \\
& \rui = \epsilon  \dd{u^1 }{x_2}(x_1,0)\left( \beta - \obeta - \tin - \tout \right) \left(\frac{x_1}{\epsilon},\frac{1}{\epsilon} \right)=:b\left(\frac{x_1}{\epsilon},\frac{1}{\epsilon} \right) \text{ on } \Gun,
 \rui = 0 , \text{ on } \Geps
\end{aligned}
\right.
\end{equation}
As $\uiue$ is only a first order approximation, a second order error remains in $\oo$.
This explains the constant source term  on the rhs of the first equation in the 
system above.
We then have
\begin{thm}\label{5.1}
\uhthms, $\rui$ satisfies
$$
\nrm{\rui}{H^1(\Oe)} \leq \epsilon 
$$
\end{thm}
\begin{proof}
We separate various sources of errors, we set $r^1$ the solution of the
Neumann part of the errors, it solves~:
\begin{equation}\label{neumann}
\left\{ 
\begin{aligned}
&- \Delta r_1 = 0, \text{ in } \Oe, \\
& \dd{r_1}{\bf n} = \dd{\xi}{\bf n}\left(\frac{L}{\epsilon},\frac{x_2}{\epsilon}\right)\text{ on } \Gout, \dd{r_1}{\bf n} =  \ddn{\tilde{\xi}}\left(\frac{L}{\epsilon},\frac{x_2}{\epsilon}\right)  \text{ on } \Gin, \\
& r_1 = 0 \quad \text{ on } \Geps \cup \Gun,
\end{aligned}
\right.
\end{equation}
then the rest $r_2$ satisfies
\begin{equation}\label{dirichlet}
\left\{ 
\begin{aligned}
&- \Delta r_2 = C \chiu{\oo} , \text{ in } \Oe, \\
& \dd{r_2}{\bf n} = 0 \text{ on } \gio, \\
& r_2 = b\left(\frac{x_1}{\epsilon},\frac{1}{\epsilon}\right) \quad \text{ on } \Gun, r_2=0  \text{ on } \Geps, 
\end{aligned}
\right.
\end{equation}
which is the Dirichlet part of the errors and should be evaluated in a second step 
thanks to appropriate extensions and lifts.

\paragraph{\bf The Neumann part}

The variational form of the problem \eqref{neumann} reads
$$
\int_{\Oe} \nabla r_1 \cdot \nabla v \, dx = \int_{\Gout} \dd{r_1}{\bf n} \gamma(v) dx_2, \quad \forall v \in H^1_D(\Oe),
$$
dividing this expression by $\nrm{\nabla v}{L^2(\Oe)}$ we first obtain the
 equivalence~:
$$
\sup_{v\in H^1_D(\Oe)} \frac{ \int_{\Oe} \nabla r_1 \cdot \nabla v \, dx }{\nrm{\nabla v}{L^2(\Oe)}} \equiv \nrm{\nabla r_1}{L^2(\Oe)}.
$$
Indeed, by Cauchy-Schwartz one has easily that the $L^2$ norm is greater than the supremum
while a specific choice of $v=r_1$ gives the reverse estimate. Thanks to this,
one has 
$$
\nrm{\nabla r_1}{L^2(\Oe)} = \sup_{v\in H^1_D(\Oe)}  \frac{ \int_{\Gout} \dd{r_1}{\bf n} \gamma(v)  dx_2 }{\nrm{\nabla v}{L^2(\Oe)}}.
$$
We underline that we kept the properties of the traces of $H^1_D(\Oe)$ functions inside
the sup that we aim to evaluate. This norm is lower that the simple $H^{-\ud}(\Gout)$ 
which authorizes different behaviors of test functions near the corners of $\Gout$ (\cite{Grisvard}, p 43, thm 1.5.2.3).

Now the integral in the rhs of the last expression reads in fact~:
$$
\int_{\Gout} \dd{r_1}{\bf n} \gamma(v)  dx_2 
= \int_{\Gout} \dd{\xi}{\bf n}\left(\frac{L}{\epsilon},\frac{ x_2 }{\epsilon} \right) \gamma(v)  dx_2 
= \epsilon \int_{f(0)}^\ue \dd{\xi}{\bf n}\left(\frac{L}{\epsilon},y_2\right) \gamma(\ti{v}) dy_2,
$$
where we constructed $\ti{v}$ as in section \ref{test} i.e. $\ti{v}$ has the same trace
as $v$ but $\ti{v}$ is expressed as a microscopic trace function. Thanks to proposition
\ref{test.prop} the corresponding microscopic trace $\ti{v}$ belongs to $\ws{\ud}{2}{0,\sigma}{\partial \Pi_{\frac{L}{\epsilon}}}$ and propositions \ref{esitm.weight.sob.trace} and \ref{test.prop}
give that
$$
\int_{f(0)}^\ue \dd{\xi}{\bf n}\left(\frac{L}{\epsilon},y_2\right) \gamma(\ti{v}) dy_2 
\leq \epsilon^{\alpha} \nrm{\ti{v}}{\ws{\ud}{2}{0,\sigma}{\partial \Pi_{\frac{L}{\epsilon}}} }
\leq K \epsilon^{\alpha} \nrm{v}{H^\ud(\partial \Oe)}
\leq  K' \epsilon^{\alpha} \nrm{v}{H^1(\Oe)}. 
$$
The same analysis and convergence rates hold on $\Gin$ with the normal derivative of $\tout$. 
All together one obtains
$$
\nrm{\nabla r_1}{L^2(\Oe)} \leq \epsilon^{1+\alpha}.
$$
\paragraph{\bf The Dirichlet part}

We should lift $b$, 
the non homogeneous Dirichlet boundary conditions on $\Gun$ defined in \eqref{sys.eq.error}, so we set
$$
\varsigma := b\left(\frac{x_1}{\epsilon},\frac{1}{\epsilon} \right) x_2^2 \chiu{\Oz},\text{ and } \ti{r}_2 := r_2 - \varsigma .
$$
Standard {\em a priori} estimates give
$$
\nrm{\nabla \ti{r}_2}{L^2(\Oe)} \leq  \nrm{\nabla \varsigma }{L^2(\Oe)} +\epsilon ,
$$
where the $\epsilon$ in the last rhs comes when estimating the constant source term $C$ localized in $\oo$, indeed~:
$$
(C,v)_{\oo} \leq \nrm{C}{L^2(\oo)} \nrm{v}{L^2(\oo)}\leq \sqrt{\epsilon} \nrm{C}{L^2(\oo)} \nrm{\nabla v}{L^2(\oo)} \leq \epsilon C \nrm{\nabla v}{L^2(\Oe)}
$$
thanks to a Poincar{\'e} inequality in the sub-layer.
Hereafter we estimate the gradient of the lift,
$$
\begin{aligned}
\nrm{\nabla \varsigma }{L^2(\Oe)} \leq& \epsilon K \nrm{\beta \left(\frac{\cdot}{\epsilon}, \frac{1}{\epsilon}\right)-\obeta}{L^2(\Gun)} 
+  K \nrm{\partial_{x_1} \beta \left(\frac{\cdot}{\epsilon}, \frac{1}{\epsilon}\right)}{L^2(\Gun)} \\
& + \epsilon \nrm{\tin}{L^2(\Gun)}  + \epsilon \nrm{\partial_{x_1}\tin}{L^2(\Gun)} + \epsilon \nrm{\tout}{L^2(\Gun)}  + \epsilon \nrm{\partial_{x_1}\tout}{L^2(\Gun)} \\
& \leq K \epsilon^{\td} +
2 \left[ \epsilon \nrm{\xi \left(\frac{\cdot}{\epsilon}, \frac{1}{\epsilon}\right)}{L^2(\Gun)}  
+  \nrm{\dd{\xi}{y_1} \left(\frac{\cdot}{\epsilon}, \frac{1}{\epsilon}\right)}{L^2(\Gun)} \right] \\
& \leq K \left[ \epsilon^{\td}  + \epsilon^{2-\frac{1}{2M}} + \epsilon^{1+\alpha} \right ]  \leq K \epsilon^{1+\alpha} 
\end{aligned}
$$
where we used the second estimate of theorem \ref{DecayXi} 
describing the decay properties of $\xi$.
This ends the proof: the main error is still made when  linearizing
the Poiseuille profile in $\oo$.
As $\rui:=r_1+r_2$ one gets the desired estimate.
\end{proof} 
Here comes the  error estimate in the $L^2$ norm that add approximately 
an $\sqrt{\epsilon}$ factor to the {\em a priori} estimates above.
\begin{thm}\label{5.2}
\uhthms,  $\rui$, the solution of system \eqref{sys.eq.error}  satisfies
$$
\nrm{\rui}{L^2(\Oz)} \leq K \epsilon^{1+\alpha},
$$
where the constant $K$ is independent on $\epsilon$ and $\alpha < \sqrt{2}/{\pi}\sim 0.45$. 
\end{thm}
\begin{proof}
We define $v$ a regular solution on the ``smooth'' domain $\Oz$ (in the sense: not rough,
in particular, $\Oz$ is a rectangle) of the  problem
$$
\left\{ 
\begin{aligned}
&-\Delta v = F ,\text{ in } \Oz, \\
& \dd{v}{\bf n}=0, \text{ on } \Gin'\cup \Gout' \\
& v=0,\text{ on } \Gz\cup \Gun
\end{aligned}
\right.
$$
where the function $F$ is in $L^2(\Oz)$. 
Thanks to \cite{Grisvard} theorem 4.3.1.4 p. 198, one has that
there are no singularities near the corners i.e. $v\in H^2(\Oz)\cap H^1_D(\Oz)$ and
$$
\nrm{v}{H^2(\Oz)} \leq K \nrm{F}{L^2(\Oz)}.
$$
Testing $\rui$ against $F$, one gets
$$
(\rui,F)_{\Oz}= \left< \dd{\rui}{\bf n}, v\right> - \left( \rui, \dd{v}{\bf n}\right)_{\Gun\cup \Geps}
$$
where the brackets stand for the duality pairing between $H^{-1}(\giop)$
and $H^1_0(\giop)$, whereas the left bracket denotes the standard $L^2(\gzu)$
scalar product.
This leads to write~:
$$
\begin{aligned}
\nrm{\rui}{L^2(\Oz)}& \leq \nrm{\ddn{\rui}}{H^{-1}(\giop)} + \nrm{\rui}{L^2(\gzu)}
& =:I_1 + I_2
\end{aligned}
$$
the latter term of the rhs is classically estimated through 
Poincar{\'e} on the sublayer and the {\em a priori}
estimates above for the $\Gz$  part~:
$$
\nrm{\rui}{L^2(\Gz)}\leq \sqrt{\epsilon} \nrm{\rui}{H^1(\oo)}\leq \sqrt{\epsilon} \nrm{\rui}{H^1(\Oe)} \leq \epsilon^\td.
$$
On $\Gun$ there is an exponentially small contribution of the periodic boundary 
layer and an almost $\epsilon^2$  term coming from the vertical correctors $\tin$
and $\tout$~:
$$
\nrm{\rui}{L^2(\Gun)}\leq K \epsilon \nrm{\xi\left(\frac{\cdot}{\epsilon},\ue\right)}{L^2(0,L)} \leq \epsilon^{2-\frac{1}{2M}}.
$$
$I_1$ follows  using the same arguments as in the proof of the {\em a priori}
estimates. The astuteness resides in the fact that 
$$
I_1 \leq \sup_{v\in H^\ud_0(\Gout)} \frac{ \left< \ddn{\rui} , v \right> }{\nrm{v}{H^\ud(\Gout)} }\leq K \epsilon^{1+\alpha}
$$
Indeed $H^1_0(\Gout')$ functions when extended by zero on $\Gout$ are
a particular subset of $H^\ud(\Gout)$ functions vanishing on $\partial \Gout$. 
At this point one uses the same estimates as in the previous proof to obtain the last term in the rhs.

\end{proof}
\subsection{The standard averaged wall law: new  error estimates}

We use the full boundary layer approximation above as an intermediate
step to prove error estimates for the wall law.
We denote $\ru:=u^\epsilon - u^1$. 
\begin{thm}\label{5.3}
\uhthms, one has
$$
\nrm{\ru}{L^2(\Oz)} \leq \epsilon^{1+\alpha}
$$
\end{thm}
\begin{proof}
We insert the full boundary layer approximation between $u^\epsilon$ and $u^1$
$$
\begin{aligned}
\ru & :=u^\epsilon - u^1 = u^\epsilon - \uiue + \uiue - u^1 \\
&= \rui + \epsilon \dd{u^1 }{x_2}(x_1,0)\left( \beta- \obeta  - \tin + \tout \right) \lrxe =: \rui +   \dd{u^1 }{x_2}(x_1,0) I_1 \\
\end{aligned}
$$
We evaluate the $L^2(\Oz)$ norm of $I_1$
$$
I_1 
\leq K \left\{ \epsilon \nrm{\beta\left(\frac{\cdot}{\epsilon}\right)-\obeta}{L^2(\Oz)} + \epsilon \nrm{\xi\left(\frac{\cdot}{\epsilon}\right)}{L^2(\Oz)} \right\} 
\leq K \left\{ \epsilon^\td + \epsilon \nrm{\xi\left(\frac{\cdot}{\epsilon}\right)}{L^2(\Oz)} \right\} 
$$
while the first term is classical and the estimate comes from \eqref{BetaL2},
for what concerns the second term,  we use the $L^\infty$ estimates of theorem \ref{DecayXi} and get
$$
\begin{aligned}
\int_{\Oz} \xi^2 \lrxe dx & = \epsilon^2 \int_0^{\frac{L}{\epsilon}} \int_0^{\ue} \xi^2 dy \leq  \epsilon^2  \int_0^{\frac{L}{\epsilon}} \int_0^{\ue} \frac{1}{\rho^{2-\frac{1}{M}}} dy \\
& \leq \epsilon^2 \sup_{\left[0,\frac{L}{\epsilon}\right]\times\left[0,\frac{1}{\epsilon}\right]} \rho^{\frac{1}{M}+\delta''} \int_{\left[0,\frac{L}{\epsilon}\right]\times\left[0,\frac{1}{\epsilon}\right]} \frac{1}{\rho^{2+\delta''}} dy \leq K \epsilon^{2-\frac{1}{M}-\delta''}
\end{aligned}
$$
where $\delta''$ is a positive constant as small as desired.
This  ends the proof.
\end{proof}

\section{Conclusion}

In this work we established  error estimates for a new boundary layer approximation and 
for the standard wall  law with respect to the exact solution of a rough problem set with 
non periodic lateral boundary conditions. The final order of approximation is of $\epsilon^{1+\alpha}$
where $1+\alpha\sim 1.45$ which is compatible and comparable to results obtained in the periodic
case (see \cite{BrMiCras,BrMiQam} and references there in).

Establishing estimates in the spirit of very weak solution \cite{Ne.Book.67} but
in the weighted context  improves the $L^2(\Oz)$ estimates 
but requires an extra amount of work not presented here.
This is done in \cite{vws}, we perform also
 a numerical validation illustrating the accuracy of our
theoretical results.

\bigskip
\bigskip
\noindent {\it Acknowledgements.}

The authors would like to thank C. {\sc Amrouche} for his advises and support, 
as well as S. {\sc Nazarov} for fruitful discussions and clarifications.
This research was partially funded by Cardiatis\footnote{www.cardiatis.com}, 
an industrial partner designing and commercializing metallic wired stents.

\bibliographystyle{plain}
\bibliography{eqred}
\end{document}